\documentclass[a4paper,12pt]{article}
\usepackage{lineno}
\usepackage{amsfonts}

\begin{document}

\newtheorem{theorem}{Theorem}[section]
\newtheorem{corollary}[theorem]{Corollary}
\newtheorem{definition}[theorem]{Definition}
\newtheorem{conjecture}[theorem]{Conjecture}
\newtheorem{question}[theorem]{Question}
\newtheorem{lemma}[theorem]{Lemma}
\newtheorem{proposition}[theorem]{Proposition}
\newtheorem{example}[theorem]{Example}
\newtheorem{problem}[theorem]{Problem}
\newenvironment{proof}{\noindent {\bf
Proof.}}{\rule{3mm}{3mm}\par\medskip}
\newcommand{\remark}{\medskip\par\noindent {\bf Remark.~~}}
\newcommand{\pp}{{\it p.}}
\newcommand{\de}{\em}

\newcommand{\JEC}{{\it Europ. J. Combinatorics},  }
\newcommand{\JCTB}{{\it J. Combin. Theory Ser. B.}, }
\newcommand{\JCT}{{\it J. Combin. Theory}, }
\newcommand{\JGT}{{\it J. Graph Theory}, }
\newcommand{\ComHung}{{\it Combinatorica}, }
\newcommand{\DM}{{\it Discrete Math.}, }
\newcommand{\ARS}{{\it Ars Combin.}, }
\newcommand{\SIAMDM}{{\it SIAM J. Discrete Math.}, }
\newcommand{\SIAMADM}{{\it SIAM J. Algebraic Discrete Methods}, }
\newcommand{\SIAMC}{{\it SIAM J. Comput.}, }
\newcommand{\ConAMS}{{\it Contemp. Math. AMS}, }
\newcommand{\TransAMS}{{\it Trans. Amer. Math. Soc.}, }
\newcommand{\AnDM}{{\it Ann. Discrete Math.}, }
\newcommand{\NBS}{{\it J. Res. Nat. Bur. Standards} {\rm B}, }
\newcommand{\ConNum}{{\it Congr. Numer.}, }
\newcommand{\CJM}{{\it Canad. J. Math.}, }
\newcommand{\JLMS}{{\it J. London Math. Soc.}, }
\newcommand{\PLMS}{{\it Proc. London Math. Soc.}, }
\newcommand{\PAMS}{{\it Proc. Amer. Math. Soc.}, }
\newcommand{\JCMCC}{{\it J. Combin. Math. Combin. Comput.}, }
\newcommand{\GC}{{\it Graphs Combin.}, }

\title{  Sharp Bounds for the Signless Laplacian Spectral Radius   in Terms of Clique
Number\thanks{ This work is supported by National Natural Science
Foundation of China (No:10971137), the National Basic Research
Program (973) of China (No.2006CB805900) and a grant of Science and
Technology Commission of Shanghai Municipality (STCSM, No:
09XD1402500).}}
\author{  Bian He, Ya-Lei Jin  and Xiao-Dong Zhang \\
{\small Department of Mathematics}\\
{\small Shanghai Jiao Tong University} \\
{\small  800 Dongchuan road, Shanghai, 200240, P.R. China}\\
{\small Email:  xiaodong@sjtu.edu.cn}
\\
Dedicated to Professors Abraham Berman, Moshe Goldberg,\\ and
Raphael Loewy  in recognition of their important contributions\\
 to
linear algebra and the linear algebra community.
 }
\date{}
\maketitle
 \begin{abstract}
 In this paper, we present  a sharp upper and lower bounds for the signless Laplacian spectral
 radius of graphs in terms of clique number. Moreover, the extremal
 graphs which attain the upper and lower bounds are characterized. In addition, these
 results
 disprove
 the two conjectures on the signless Laplacian spectral radius in [P.~Hansen and C.~Lucas, Bounds and conjectures for
 the signless Laplacian index of graphs, Linear Algebra Appl., 432(2010)
 3319-3336].
 \end{abstract}


\vskip 0.5cm

\section{Introduction}

 Throughout this paper,  we only consider  simple and undirected graphs.
  Let $G = (V(G),~E(G))$ be a simple graph with vertex set
$V(G)=\{v_1,\cdots, v_n\}$ and edge set $E(G)$.  Let $A(G)=(a_{ij})$
be the $(0,1)$ {\it  adjacency matrix} of $G$ with $a_{ij}=1$ for
$v_i$ adjacent to $v_j$ and $0$ otherwise. Moreover, let
$D(G)=diag(d(u), u\in V)$ be the diagonal matrix of vertex degrees
$d(u)$ of $G$. Then  $Q(G)=D(G)+A(G)$ is called {\it the signless
Laplacian matrix }
 of $G$.  The signless Laplacian matrix $Q(G)$ can be
  viewed as an operator on the space of functions
 $f:V(G)\rightarrow \mathcal{R}$ which satisfies
 $$Q(G)f(u)=\sum_{v\sim u}(f(u)+f(v)),$$
where $"\sim"$ stands for the adjacency relation.
 The largest eigenvalue of $Q(G)$ is called the {\it signless
Laplacian spectral radius} of $G$ and denoted by $q_1(G)$, or for
short $q_1$.

 Recently, the signless Laplacian matrix of a graph has received
increasing attention. For example, Desai and Rao \cite{desai1994}
used the smallest eigenvalue of the signless Laplacian matrix of a
connected graph to serve as a measure of how much a graph is close
to bipartite, since 0 is the smallest eigenvalue of $Q(G)$ if and
only if G is bipartite.  Liu and Liu  \cite{liu2008} presented lower
and/or upper bounds for the clique number and the independence
number in terms of the signless Laplacian eigenvalues. Oliveira et.
al. \cite{Oliveira2010} gave  several upper  and lower bounds for
the signless Laplacian spectral radius. Zhang \cite{zhang2008,
zhang2009} investigated the largest signless Laplacian spectral
radius for a given degree graphic sequence.
 Cvetovi\'{c} and
 Simi\'{c} \cite{cvetkovic2009}-\cite{cvetkovic2010-2}
surveyed  spectral graph theory based on the signless Laplacian
matrix of a graph.  Recently, Hansen and Lucas \cite{hansen2010}
proposed some conjectures on the signless Laplacian spectral radius.
\begin{conjecture}\label{con1}
(\cite{hansen2010}) Let $G$ be a connected graph on $n\ge 4$
vertices with signless Laplacian spectral radius $q_1$ and clique
number $\omega$. Then
\begin{equation}\label{con1-1}
q_1-\omega\le \frac{3}{2}n-4, {\rm if}\ n\  {\rm is \ \ even},
\end{equation}
\begin{equation}\label{con1-2}
\frac{q_1}{\omega}\le \frac{n}{2}.
\end{equation}
The bound for (\ref{con1-1}) is attained by and only by the
complement of a perfect matching when $n\ge 6$ is even. Moreover,
when $n\ge 9$ is odd, $q_1-\omega$ is maximum for and only for the
complement of a perfect matching on $n-3$ vertices and a triangle on
three remaining vertices. The bound for (\ref{con1-2}) is attained
by and only by the complete bipartite graph $K_{p,q}$.
\end{conjecture}

\begin{conjecture}(\cite{hansen2010})\label{con2}
Let $G$ be a connected graph on $n\ge 6$ vertices with signless
Laplacian spectral radius $q_1$ and chromatic number $\chi$. Then
\begin{equation}\label{con2-1}
q_1-\chi\le \frac{3}{2}n-4, \ \ \ {\rm if }\ n \ {\rm is \ \ even.}
\end{equation}
The bound is attained by and only by the complement of a perfect
matching when $n$ is even. Moreover, when $n\ge 9$ is odd,
$q_1-\chi$ is maximum for and only for the complement of a perfect
matching on $n-3$ vertices and a triangle on the three remaining
vertices.
\end{conjecture}

Hansen and Lucas  \cite{hansen2009} obtained some results on
signless Laplacian related to clique number and chromatic number.
For other results, see \cite{cvetkovic2009}-\cite{cvetkovic2010-2}
and the references therein.

On the other hand, there are many {\it Tur\'{a}n-type extremal
problems,} i.e., given a forbidden graph $H$, determine the maximal
number of edges in a graph on $n$ vertices that does not contain a
copy of $H$.  It states that among n-vertex graphs not containing a
clique of size $t+1$, the complete $t$-partite graph $T_{n, t}$ with
(almost) equal parts, which is called {\it Tur\'{a}n graph,} has the
maximum number of edges. Spectral graph theory has similar Tur\'{a}n
extremal problems which determine the largest (or smallest)
eigenvalue of  a graph not containing a subgraph $H$.

 Nikiforov \cite{nikiforov2007} proved a spectral extremal
Tur\'{a}n theorem:  let $\lambda(G)$ be the largest eigenvalues of
the adjacency matrix of $G$ not containing  complete graph $K_t$ of
order $t$ as a subgraph, then $\lambda(G)\le \lambda(T_{n,t-1})$
with equality if and only if $G=T_{n, t-1}$;  the same result has
been proved previously by Guiduli \cite{guiduli1998},
  but his proof was made public only after the publication of \cite{nikiforov2007}.
Further,  Nikiforov explicitly advocated the study of general
Tur\'{a}n problems in many publications (see
\cite{nikiforov2007}-\cite{nikiforov2011}).  For instance, he
determined \cite{nikiforov2010-1} the maximum spectral radius of
graphs without paths of given length  and presented
\cite{nikiforov2011} a comprehensive survey on these topics.
  In addition, Sudakov et al. \cite{sudakov2005} presented a generalization of Tur\'{a}n
Theorem in terms of Laplacian eigenvalues.

 Motivated by these conjectures and Tur\'{a}n-type  extremal problems,  we
 investigate in
 this paper the extremal graphs  with maximal or minimal signless Laplacian
 spectral radius among all graphs of order $n$ with given clique
 number, which may be regarded as a part of spectral extremal theory.
The main result of this paper reads:
 \begin{theorem}\label{thm1}
 Let $G$ be a connected graph of order $n$ with  clique number
 $\omega\ge 2$. Then
 \begin{equation}\label{thm1-1}
 q_1(G)\le
 \frac{(3\omega-4)k+3r-2+\sqrt{k^2\omega^2+[(2r+4)\omega-8r]k+(r-2)^2}}{2},
\end{equation}
where $n=k\omega+r, 0\le r<\omega$.  Moreover,  equality holds in
(\ref{thm1-1}) if and only if $G$ is complete bipartite graph for
$\omega=2$ and   Tur\'{a}n graph $T_{n, \omega}$ for $\omega\ge 3$.
\end{theorem}

\section{Proof of Theorem~\ref{thm1}}
In order to prove Theorem~\ref{thm1}, we need more notation and
preliminary results. Two vertices in a graph $G$ are called {\it
duplicate} if they have precisely the same neighbors. In the
operation of {\it duplication of a vertex $u$ to a vertex $v$,} all
edges incident to vertex $u$ are deleted and then  all edges between
$u$ and vertices adjacent to $v$ are added. Clearly, after this
duplication, $u$ and $v$ are duplicate, and duplication in a graph
not containing  $K_{t+1}$ preserves this property.  The {\it
complement} $\overline{G}$ of a simple graph $G=(V(G), E(G))$ is the
simple graph with vertex set $V(G)$, two vertices being adjacent in
$\overline{G}$ if and only if they are not adjacent in $G$. For two
disjoint graphs $ G$ and $H$, the {\it union} $G+H$ of $G$ and $H$
is the graph with vertex set $V(G)\bigcup V(H)$ and edge set
$E(G)\bigcup E(H)$; the {\it join} $G\bigtriangledown H$ of $G$ and
$H$ is the graph obtained from the union $G+H$ by joining each
vertex of $G$ to each vertex of $H$.

Let $f$ be a nonnegative function on vertex set $V(G)$ of $G$. Then
$f$ is also regarded as a nonnegative vector corresponding to vertex
set $V(G)$. For any vertex $u$, the {\it weight} $w_G(u)$ of vertex
$u$ with respect to $f$ in $G$ is defined as
\begin{equation}\label{weight-def}
w_G(u)=Q(G)f(u)=\sum_{v\sim u}(f(u)+f(v)) \end{equation}
 if $u$ is
not isolated vertex and $0$ otherwise, where $"\sim"$ stands for
adjacency relation.
 We adapt a variation of Zykov's proof
 of Tur\'{a}n's theorem, which was also used in \cite{guiduli1998}.
\begin{lemma}\label{lemma2-1}
Let $G$ be a simple  graph of order $n$  not containing  $K_{t+1}$.
Then there exists a simple graph $G_1$ such that $G_1=\overline{K_p}
\bigtriangledown  H$ and $q_1(G)\le q_1(G_1)$, where $H$ is a simple
graph not containing  $K_{t}$ and $1\le p\le n-1.$ Moreover,
equality holds if and only if $G=G_1$.
\end{lemma}
\begin{proof} Without loss of generality, we assume that $G$ is
connected, since $q_1(G)$ is an increasing function with respect to
adding edges.  Let $f$ be the unit positive eigenvector on $V(G)$
corresponding to the eigenvalue $q_1(G)$ of $G$ such that
$q_1(G)=\langle f, \ Q(G)f \rangle$ and $\langle f,\ f\rangle =1$.
 Let $w_G(u)=\max\{w_G(v):\ v\in V(G)\}$. Denote by $V_1$
the set of all neighbors of vertex $u$ and $V_2=V(G)\backslash
(V_1\bigcup\{u\})$ with $|V_2|=p-1$.  Now we construct a new graph
$G_1$ obtained from $G$ by a duplication of
 each vertex $v\in V_2$ to vertex $u$.
 Then  $G_1$ still does not contain
 $K_{t+1}$ and the induced subgraph $G[V_1]$ does not  contain $K_{t}$.
 Hence $G_1$ can be written as $G_1=\overline{K_p}\bigtriangledown H$, where
 $H=G[V_1]$ does not contain $K_{t}$.

  On the other hand, for any $v\in V_1$,
  any vertex
 adjacent to $v$ in $G$ must be adjacent to $v$ in $G_1$. Then
 \begin{equation}\label{lemma2-1-1}
 w_{G}(v)=\sum_{x v\in E(G)}(f(v)+f(x))\le \sum_{xv\in
 E(G_1)}(f(v)+f(x))=w_{G_1}(v)
\end{equation}
with equality  holding if and only if $v$ is adjacent to  each
vertex in $V_2$ in $G$.
 For any vertex $v\in V_2$,
 \begin{equation}\label{lemma2-1-2}q_1(G)f(u)=Q(G)f(u)=w_G(u)\ge
 w_G(v)=Q(G)f(v)=q_1(G)f(v).
 \end{equation}
 Hence $f(u)\ge f(v)$.  In addition,
$$w_G(u)=q_1(G)f(u)=\sum_{xu\in E(G)}(f(u)+f(x))=d_G(u)f(u)+\sum_{x\in
V_1}f(x)$$ and
 $q_1(G)>d_G(v)$ for any $v\in
 V(G)$ (see \cite{Oliveira2010}). Therefore, for any $v\in V_2$,
 \begin{eqnarray*}
 w_{G_1}(v) &=& \sum_{xv\in E(G_1)}(f(v)+f(x))=d_{G_1}(v)f(v)+\sum_{x\in V_1}f(x)\\
  &=& d_G(u)f(v)+w_G(u)-d_G(u)f(u)\\
  &=&w_G(v)+(w_G(u)-w_G(v))-d_G(u)(f(u)-f(v))\\
  &=&w_G(v)+(q_1(G)-d_G(v))(f(u)-f(v))\\
  &\ge& w_G(v).
 \end{eqnarray*}
Moreover, $w_{G_1}(u)=w_G(u)$. Then
 $$ \langle f, Q(G_1)f  \rangle=\sum_{v\in V(G_1)}f(v)w_{G_1}(v)\ge \sum_{v\in V(G)}f(v)w_G(v)=q_1(G).$$
 Hence by Rayleigh quotient, $q_1(G_1)\ge q_1(G)$ with equality holding if and only if
 $w_{G_1}(v)=w_{G}(v)$ for all $v\in V(G)=V(G_1)$, which implies $G=G_1$.
  \end{proof}

\begin{lemma}\label{lemma21}
Let $G$ be a simple graph not containing $K_{t+1}$ and $R$ be
another simple  graph. Then there exists a simple graph $G_1$ such
that $G_1=\overline{K_p}\bigtriangledown H$ and
$q_1(R\bigtriangledown G)\le q_1(R\bigtriangledown
G_1)=q_1(R\bigtriangledown \overline{K_p}\bigtriangledown H)$, where
$H$ is a simple graph not containing $K_{t}$ and $1\le p\le n-1.$
\end{lemma}
\begin{proof}
Let $f$ be unit positive eigenvector of $Q(R\bigtriangledown G)$
corresponding to $q_1(R\bigtriangledown G)$. Let
$$w_{R\bigtriangledown
 G}(v)\equiv Q(R\bigtriangledown G)f(v)=\sum_{xv\in E(R\bigtriangledown G)}(f(v)+f(x))
=q_1(R\bigtriangledown G)f(v)
 $$
for any $v\in V(R\bigtriangledown G)$, and let $w_{R\bigtriangledown
G}(u)=\max\{w_{R\bigtriangledown
 G}(v):\ v\in V(G)\}$. Therefore, as in (\ref{lemma2-1-2}),
 $$q_1(R\bigtriangledown G)f(u)=\sum_{xu\in E(R\bigtriangledown  G)}(f(u)+f(x))
 =w_{R\bigtriangledown
G}(u)\ge w_{R\bigtriangledown G}(v)= q_1(R
 \bigtriangledown
 G)f(v)$$ which implies  $f(u)\ge f(v)$ for any $v\in V(G)$.
Denote by $V_1$ the set of all neighbors of vertex $u$ in $V(G)$ and
$V_2=V(G)\backslash (V_1\bigcup\{u\})$ with $|V_2|=p-1$. Now we
construct a new graph $G_1$ obtained from $G$ by a duplicatation of
 each vertex $v\in V_2$ to vertex $u$.
 Then  $G_1$ still does not contain
 $K_{t+1}$ and the induced subgraph $G[V_1]$ does not contain $K_{t}$.
 Hence $G_1$ can be written as $G_1=\overline{K_p}\bigtriangledown H$, where
 $H=G[V_1]$ does not contain $K_{t}$.

Clearly, for any vertex  $v\in V(R)$, $w_{R\bigtriangledown
G_1}(v)=w_{R\bigtriangledown  G}(v)$.  For any vertex $v\in V_1$,
any neighbor of $v$ in $R\bigtriangledown  G$ is also its neighbor
in $R\bigtriangledown  G_1$. So for any $v\in V_1$,
$w_{R\bigtriangledown G_1}(v)\ge w_{R\bigtriangledown  G}(v)$ with
equality if and only if $v$ is adjacent to each vertex in $V_2$ in
$G$. For any vertex $v\in V_2$, by simple calculations, as in the
previous proof,  we have
\begin{eqnarray*} w_{R\bigtriangledown  G_1}(v) &=&
\sum_{xv\in E(G_1)}(f(v)+f(x))+\sum_{x\in V(R)}(f(v)+f(x))\\
&=&\sum_{x\in V_1}((f(u)+f(x))-(f(u)-f(v)))+\sum_{x\in
V(R)}((f(u)+f(x))-(f(u)-f(v)))\\
&=&\sum_{x\in V_1}(f(u)+f(x))+\sum_{x\in
V(R)}(f(u)+f(x))-(|V_1|+|V(R)|)(f(u)-f(v))\\
&=&w_{R\bigtriangledown G}(u)-(|V(R)|+d_G(u))(f(u)-f(v))\\
&=&w_{R\bigtriangledown G}(v)+(w_{R\bigtriangledown
G}(u)-w_{R\bigtriangledown G}(v))-(|V(R)|+d_G(u))(f(u)-f(v))\\
&=&w_{R\bigtriangledown G}(v)+(q_1(R\bigtriangledown
G)f(u)-q_1(R\bigtriangledown G)f(v))-(|V(R)|+d_G(u))(f(u)-f(v))\\
 &=&w_{R\bigtriangledown  G}(v)+(q_1(R\bigtriangledown
G)-(|V(R)|+d_G(u)))(f(u)-f(v))\\
&\ge & w_{R\bigtriangledown G}(v),
\end{eqnarray*}
since  $q_1(R\bigtriangledown G)>|V(R)|+d_G(u)$ by Perron-Frobenius
theorem. Moreover, $w_{R\bigtriangledown
G_1}(u)=w_{R\bigtriangledown G}(u)$. Therefore
\begin{eqnarray*}
q_1(R\bigtriangledown G)&=&  \langle f,\ Q(R\bigtriangledown G)f
\rangle =\sum_{v\in V(R\bigtriangledown G)}f(v)
w_{R\bigtriangledown G}(v)\\
& \le& \sum_{v\in V(R\bigtriangledown G_1)}f(v) w_{R\bigtriangledown
G_1}(v)=  \langle f,\ Q(R\bigtriangledown
  G_1)f \rangle\\
  &\le&  q_1(R\bigtriangledown G_1).
  \end{eqnarray*}
Moreover, equality holds if and only if $f$ is an eigenvector of
$Q(R \bigtriangledown G_1)$, which implies $w_{R\bigtriangledown
G}(v)=w_{R\bigtriangledown  G_1}(v) $ for all $v\in
V(R\bigtriangledown  G)$. So equality holds if and only if $G=G_1$.
Hence the assertion holds.
\end{proof}

  The following lemma  on the signless Laplacian spectral radius
  of  a complete $t-$partite graph is well-known (see, for example,
   \cite{cai2009} or \cite{yu2012}).

\begin{lemma}(\cite{cai2009},\cite{yu2012})\label{lemma2-2}
Let $G$ be a complete $t-$partite graph on $n$ vertices. If $n=tk+r,
0\le r< t$, then $q_1(G)=n$ for $t=2$, and
 \begin{equation}\label{lemma2-2-1}
 q_1(G)\le \frac{(3t-4)k+3r-2+\sqrt{t^2k^2+[(2r+4)t-8r]k+(r-2)^2}}{2}
\end{equation}
for $t\ge 3$. Moreover, equality holds if and only if $G$ is
Tur\'{a}n graph $T_{n, t}$ which is the complete $t$-partite graph
on $n$ vertices in which the partite sets are  of size  $k$ or
$k+1$.
\end{lemma}
\begin{corollary}\label{cory4} For any $2\le t\le n-1$,
$$q_1(T_{n, t})< q_1(T_{n, t+1}).$$
\end{corollary}
\begin{proof}
It follows directly from Lemma~\ref{lemma2-2} with a simple
calculation.
\end{proof}

 Now we are ready to present a proof of Theorem~\ref{thm1}.

\begin{proof}
Let $\omega(G)=t$ and $f$ be the positive  eigenvector of $Q(G)$
corresponding to $q_1(G)$ with $ \langle f, f \rangle =1.$  We
consider the following two cases.

{\bf Case 1:} $t=2$. Then by \cite{cvetkovic2010}, $G$ is bipartite
and $q_1(G)\le n$ with equality if and only if $G$ is complete
bipartite graph. Hence (\ref{thm1-1}) holds with equality if and
only if $G$ is complete bipartite graph.

{\bf Case 2:} $t\ge 3$. Then $G$ does not contain $K_{t+1}$ as a
subgraph. By Lemma~\ref{lemma2-1}, there exists a graph
$G_1=\overline{K_{n_1}}\bigtriangledown H_1$ such that $H_1$ does
not contain $K_{t}$ and $q_1(G)\le q_1(G_1)$. Moreover, equality
holds if and only if $G=G_1$.  Since $\overline{K_{n_1}}$ is a
simple graph and $H_1$ does not contain $K_{t}$,  by
Lemma~\ref{lemma21}, there exists a graph
$G_2=\overline{K_{n_1}}\bigtriangledown
\overline{K_{n_2}}\bigtriangledown H_2$ such that $H_2$ does not
contain $K_{t-1}$ and $q_1(G_1)\le q_2(G_2)$. Moreover, equality
holds if and only if $G_2=G_1$. If $H_2$ does not contain any edges,
then $H_2=\overline{K_{n_3}}$. Hence
$G_2=\overline{K_{n_1}}\bigtriangledown
\overline{K_{n_2}}\bigtriangledown \overline{K_{n_3}}$ is complete
3-partite graph and  $q_1(G)\le q_1(G_1)\le q_1(G_2)=
q_1(\overline{K_{n_1}}\bigtriangledown
\overline{K_{n_2}}\bigtriangledown \overline{K_{n_3}})$. If $H_2$
contains at least one edge, then by Lemma~\ref{lemma21}, there
exists a graph $G_3$ such that
$G_3=\overline{K_{n_1}}\bigtriangledown\overline{K_{n_2}}
\bigtriangledown
  \overline{K_{n_3}} \bigtriangledown H_3$ and
 $q_1(G_2)\le q_1(G_3)$ with $H_3$ not containing $K_{t-2}$, since
 $H_2$
does not contain $K_{t-1}$ and $\overline{K_{n_1}}\bigtriangledown
\overline{K_{n_2}}$ is a simple graph.
 By repeated use of Lemma~\ref{lemma21},
there exists a series of graphs $G_1,\cdots, G_s$ such that
$G_i=\overline{K_{n_1}}\bigtriangledown\overline{K_{n_2}}
\bigtriangledown
 \cdots\bigtriangledown \overline{K_{n_i}}\bigtriangledown H_i$ and
 $q_1(G)\le q_1(G_1)\le\cdots\le q_1(G_i)$ with equality if
and only if $G_i=G$, where $H_i$ does not contain $K_{t+1-i}$ and
$i=1,\cdots, s\le t$. Moreover, $H_s=\overline{K_{n_s}}$. Therefore,
$G_s$ is a complete s-partite graph. Further, by
Lemma~\ref{lemma2-2} and Corollary~\ref{cory4},
\begin{eqnarray*}
q_1(G)&\le& q_1(G_1)\le \cdots\le  q_1(G_s)\le q_1(T_{n, s})\le
q_1(T_{n,t})\\
&=& \frac{(3t-4)k+3r-2+\sqrt{t^2k^2+[(2r+4)t-8r]k+(r-2)^2}}{2}
\end{eqnarray*}
 with equality if and only if $G=G_1=\cdots =G_s=T_{n,s}=T_{n, t}$.
 This completes the proof.
\end{proof}

{\bf Remark:} From Theorem~\ref{thm1}, we are able to deduce
Tur\'{a}n theorem for $t\ge 3$.
\begin{corollary}\label{cor2-1}
Let $G$ be a connected graph of order $n$ not containing $K_{t+1}$.
If $t\ge 3$ and $n=kt+r, 0\le r<t$, then
$$|E(G)|\le |E(T_{n,t})|= \frac{t^2-t}{2}k^2+(t-1)rk+\frac{r(r-1)}{2}.$$
Moreover, equality holds if and only if $G=T_{n,t}$.
\end{corollary}
\begin{proof}
Clearly $|E(T_{n,t})|= \frac{t^2-t}{2}k^2+(t-1)rk+\frac{r(r-1)}{2}.$
 By
Rayleigh's quotient, it is easy to see that $q_1(T_{n,t})\ge
\frac{4|E(T_{n,t})|}{n}$. Then
$\frac{nq_1(T_{n,t})}{4}-|E(T_{n,t})|\ge 0$ and
 \begin{eqnarray*}
 \varepsilon& \equiv &
\frac{nq_1(T_{n,t})}{4}-|E(T_{n,t})|\\
&=&-\frac{1}{8}\{ t^2k^2+[(2r+2)t-4r]k+r^2-2r-\\
&&(tk+r)\sqrt{t^2k^2+[(2r+4)t-8r]k+(r-2)^2}\}\\
& \ge & 0. \end{eqnarray*} On the other hand, let
$$\varphi(x)\equiv 4x^2+\{t^2k^2+[(2r+2)t-4r]k+r^2-2r\}x+(r-t)rk^2+(r-t)rk.$$
Then  $\varphi(x)=0$ has a root
$$x_1=-\frac{1}{8}\{ t^2k^2+[(2r+2)t-4r]k+r^2-2r-
(tk+r)\sqrt{t^2k^2+[(2r+4)t-8r]k+(r-2)^2}\}=\varepsilon,$$ i.e.,
$\varepsilon$ is a nonnegative root of $\varphi(x)=0$. Further,
 since $$\varphi(0)=(r-t)rk^2+(r-t)rk\le 0$$ and
 $$\varphi(1)=
(t^2-rt+r^2)k^2+[(r+2)t+r^2-4r]k+(r-1)^2>0,$$
 equation
$\varphi(x)=0$ has only two roots: one is  negative and the other is
 nonnegative  which lies in $[0,1)$. Then $0\le
\varepsilon<1$, i.e., $\lfloor \varepsilon\rfloor=0$, where $\lfloor
a\rfloor$ is the largest integer not greater than $a$. Hence
 $\lfloor
\frac{nq_1(T_{n,t})}{4}\rfloor=|E(T_{n,t})|$. Therefore, by
Theorem~\ref{thm1},
$$|E(G)|\le \lfloor\frac{nq_1(G)}{4}\rfloor\le
\lfloor\frac{nq_1(T_{n,t})}{4}\rfloor=|E(T_{n,t})|.$$ So the
assertion holds.
\end{proof}

The following result has been proved in \cite{cai2009} and
\cite{yu2012}.

\begin{corollary}(\cite{cai2009}, \cite{yu2012})\label{cor2}
Let $G$ be a connected simple graph of order $n$ with  chromatic
number $\chi\ge 3$. If $n=k\chi+r, 0\le r<\chi,$ then
$$q_1(G)\le
 \frac{(3\chi-4)k+3r-2+\sqrt{k^2\chi^2+[(2r+4)\chi-8r]k+(r-2)^2}}{2}$$
 with equality if and only if $G$ is $T_{n,\chi}$.
 \end{corollary}
 \begin{proof} The assertion  follows from $\omega(G)\le \chi(G)$ and
Theorem~\ref{thm1}.
\end{proof}

The following corollary gives some conditions under which
Conjectures~\ref{con1} and \ref{con2} hold, or need not hold.
\begin{corollary}\label{cor3}
Let $G$ be a  simple connected graph of order $n\ge 10$ with clique
number $\omega$.

(i). If $\omega\le 4$ or $\omega\ge \lceil\frac{n}{2}\rceil$, then
$q_1(G)\le \frac{3n}{2}+\omega-4$. Moreover, equality holds if and
only if $G$ is Tur\'{a}n graph $T_{n,4}$ and $n=4k$ or Tur\'{a}n
graph $T_{n,k}$ and $n=2k$. In other words, Conjectures~\ref{con1}
and \ref{con2} hold.

(ii). If $5\le \omega<\lceil\frac{n}{2}\rceil$, then
$q_1(T_{n,\omega})>\frac{3n}{2}+\omega-4$. In other words,
Conjectures~\ref{con1} and \ref{con2} generally do not hold.

(iii). $\frac{q_1(G)}{\omega}\le\frac{n}{2}.$
\end{corollary}
\begin{proof}
If $\omega=2$, then $q_1(G)\le n$ and $q_1(G)\le n\le
\frac{3n}{2}+2-4$. So (i) holds.

  If $\omega=3$ and $n=3k+r$ with $0\le r<3$, then by
  (\ref{thm1-1}),
\begin{eqnarray*}
q_1(G)&\le & \frac{5k+3r-2+\sqrt{9k^2+(12-2r)k+(r-2)^2}}{2}\\
&\le &\frac{5k+3r-2+3k+2}{2}\\
&\le & \frac{9k+3r-2}{2}= \frac{3n}{2}+\omega-4.
\end{eqnarray*}
 So (i) holds.

 If
$\omega= 4$,  by (\ref{thm1-1}), we have
$$q_1(G)\le \frac{8k+3r-2+\sqrt{16k^2+16k+(r-2)^2}}{2}\le \frac{8k+3r-2+(4k+2)}{2}= \frac{3n}{2}+\omega-4$$
with equality if and only if  $n=4k$. Hence (i) holds.

Now assume that  $\omega\ge\lceil \frac{n}{2}\rceil$. If
$n=2\omega$, then by (\ref{thm1-1}),
$$q_1(G)\le
\frac{2(3\omega-4)-2+\sqrt{4\omega^2+8\omega+4}}{2}=4\omega-4=\frac{3n}{2}+\omega-4.$$
Moreover, equality if and only if $n=2\omega$. If $n\neq 2\omega,$
 then $n=\omega + r$ with $r<\omega$ and
\begin{eqnarray*}
q_1(G)&\le &
\frac{(3\omega-4)+(3r-2)+\sqrt{\omega^2+(2r+4)\omega-8r+(r-2)^2}}{2}\\
&< & \frac{(3\omega-4)+(3r-2)+2\omega-2}{2}\\
&=&\frac{3n}{2}+\omega-4.
\end{eqnarray*}
Hence (i) holds.

(ii). Suppose that $5\le \omega< \lceil \frac{n}{2}\rceil$.
 Let $n=k\omega+r$, $0\le r<\omega$. Then $r\neq 0$ or  $k> 2$. Hence
  by Theorem~\ref{thm1} and some calculations, it is easy to verify
  that
  \begin{eqnarray*}q_1(T_{n, \omega})
  &=  &\frac{(3\omega-4)k+3r-2+\sqrt{\omega^2k^2+[(2r+4)\omega-8r]k+(r-2)^2}}{2}\\
  &\ge &\frac{(3\omega-4)k+3r-2+(k\omega+2)}{2}\\
  &=& \frac{3n}{2}+\frac{(\omega-4)k}{2}\\
  &\ge &\frac{3n}{2}+\omega-4.
  \end{eqnarray*}
  Moreover, if  $q_1(T_{n, \omega})=\frac{3n}{2}+\omega-4$, then $k=2$ and $r=0$ which
  is impossible.  So (ii) holds.

   (iii). If $\omega=1$, then $q_1(G)=0<\frac{n\omega}{2}$. If $\omega\ge
   2$, then by \cite{cvetkovic2010}, $q_1(G)\le
   n=\frac{n\omega}{2}$. If $\omega=3$, the assertion follows from
   (1). If $\omega\ge 4$, then by \cite{cvetkovic2010}, $q_1(G)\le
   2n-2<\frac{n\omega}{2}$. Hence (iii) holds.
\end{proof}

{\bf Remark: } In fact, Theorem~\ref{thm1} also confirms the
following conjecture of Hansen and Lucas \cite{hansen2009}.
\begin{corollary}\label{conj-3}(\cite{hansen2009})
Let $G$ be a graph of order $n$ with clique number $\omega$. Then
$q_1(G)\le 2n(1-\frac{1}{\omega})$. Moreover, if $\omega\neq 2$, the
upper bound is sharp if and only if $G$ is a complete regular
$\omega$-partite graph.
\end{corollary}
\begin{proof} If $\omega=1$, then $q_1(G)=0$ and  the assertion
holds. Now assume that $\omega\ge 2$ and  $n=k\omega+r$ with $0\le
r<\omega$. Then  $(r-2)^2\le (r+2-\frac{4r}{\omega})^2$ with
equality if and only if $r=0$. Hence by (\ref{thm1-1}),
\begin{eqnarray*} q_1(G) &\le &
\frac{(3\omega-4)k+3r-2+\sqrt{k^2\omega^2+[(2r+4)\omega-8r]k+(r-2)^2}}{2}\\
&\le
&\frac{(3\omega-4)k+3r-2+\sqrt{k^2\omega^2+2k\omega(r+2-\frac{4r}{\omega})+
(r+2-\frac{4r}{\omega})^2}}{2}\\
&= & \frac{(3\omega-4)k+3r-2+k\omega+r+2-\frac{4r}{\omega}}{2}\\
&=& \frac{4k\omega+4r-4k-\frac{4r}{\omega}}{2}\\
&=& 2n(1-\frac{1}{\omega}).
\end{eqnarray*}
So the assertion holds.
\end{proof}

\section{Further Results}

In this section, we begin with  a lower bound for the signless
Laplacian spectral radius of a graph in terms of clique number.
\begin{theorem}\label{them2}
 Let $G$ be a connected graph of order $n$ with clique number
 $\omega$.

 (i). If $\omega=2$, then $q_1(G)\ge 2+2\cos\frac{\pi}{n}$ with
 equality if and only  if $G$ is a path of order $n$.

 (ii). If $\omega\ge 3,$ then
 \begin{equation}\label{cor3-1-1}
 q_1(G)\ge q_1(Ki_{n,\omega}),
 \end{equation}
  where $Ki_{n,\omega}$ is the {\it kite graph} of order $n$ which is  obtained
  by joining one vertex of a complete graph
  $K_{\omega}$  to one end vertex of a path $P_{n-\omega}$  with a bridge.
  Moreover, equality holds if and only if $G$ is the kite
  $Ki_{n,\omega}$.
\end{theorem}
\begin{proof}  If $\omega=2$, (i) follows from Lemma~1 in \cite{hansen2010}.
 Now assume that $\omega\ge 3$. Since $q_1(G)$ is an strictly
 increasing function with respect to adding edges for a connected
 graph, there exists a connected graph $G_1$ of order $n$ obtained from $G$
 by deleting some edges such that $\omega(G_1)=\omega$ and any
 proper subgraph of $G_1$ is disconnected  or the clique number
 fewer than $\omega$. Hence $q_1(G)\ge q_1(G_1)$ with equality if
 and only if $G_1=G$. By repeated use of Theorem~2.2 in
 \cite{cvetkovic2010}, $q_1(G_1)\ge q_1(Ki_{n,\omega})$ with
 equality if and only if $G_1=Ki_{n,\omega}$. Hence the assertion
 holds.
\end{proof}

\begin{corollary}\label{cor4}
Let $G$ be a connected graph of order $n$ with clique number
$\omega\ge 3$. Then \begin{equation} q_1(G)\ge
\frac{2\omega-1+\sqrt{4\omega^2-12\omega+17}}{2} \end{equation}
 with
equality if and only if $G=Ki_{n,n-1}$ and $\omega=n-1$.
\end{corollary}
\begin{proof} By Theorem~\ref{them2}, $q_1(G)\ge q_1(Ki_{n,\omega}).$
Since the line graph of $Ki_{\omega+1,\omega}$ is the induced
subgraph of the line graph of $Ki_{n,\omega}$, it is easy to see
that $q_1(Ki_{n,\omega})\ge q_1(Ki_{\omega+1,\omega})$ with equality
if and only if $\omega=n-1$.  Hence by~\cite{hansen2010}, we have
$$q_1(Ki_{n,\omega})\ge
q_1(Ki_{\omega+1,\omega})=\frac{2\omega-1+\sqrt{4\omega^2-12\omega+17}}{2}.$$
Hence the assertion holds.
\end{proof}

{\bf Remark:} Oliveira et.al. \cite{Oliveira2010} presented several
sharp upper bounds for the signless Laplacian spectral radius of a
graph  in terms of vertex degrees and 2-average degree. Let $G$ be a
graph with degree sequence $d_1\ge d_2\ge \cdots\ge d_n$ and
$m_i=\frac{\sum_{v_jv_i\in E(G)}d_j}{d_i}$ for $i=1, \cdots, n$.
Part results can be stated as follows:

\begin{equation}\label{Oli158}
q_1(G)\le \max_{1\le i\le n}\{\ \ d_i+\sqrt{d_im_i} \  \ \}
\end{equation}
and
\begin{equation}\label{Oli158-2}
q_1(G)\le \max_{1\le i\le n} \left\{\
 \ \frac{d_i+\sqrt{d_i^2+8d_im_i}}{2}\ \ \right\}.
\end{equation}
Liu and Liu \cite{liu2008} gave an upper bound in terms of the
largest degree and clique number, i.e.,
\begin{equation}\label{Liucmj}
q_1(G)\le n+d_1-\frac{n}{\omega(G)}.
\end{equation}
Yu et.~al. \cite{yu2011} obtained the following upper bound in terms
of degree sequence, i.e.,
\begin{equation}\label{shu434}
q_1(G)\le \min_{1\le i\le n }\left\{\
\frac{d_1+2d_i-1+\sqrt{(2d_i-d_1+1)^2+8(i-1)(d_1-d_i)}}{2}\
\right\}.
\end{equation}

In general, these bounds (\ref{thm1-1}), (\ref{Oli158}),
(\ref{Oli158-2}) (\ref{Liucmj}), (\ref{shu434}) are not comparable.
We present two examples to illustrate that our bounds are best in
some cases.

\begin{example} Let $G_1$ be Tur\'{a}n graph $T_{10,3}$ of order $10$
 and $G_2$ be a graph of order $7$ as follows:

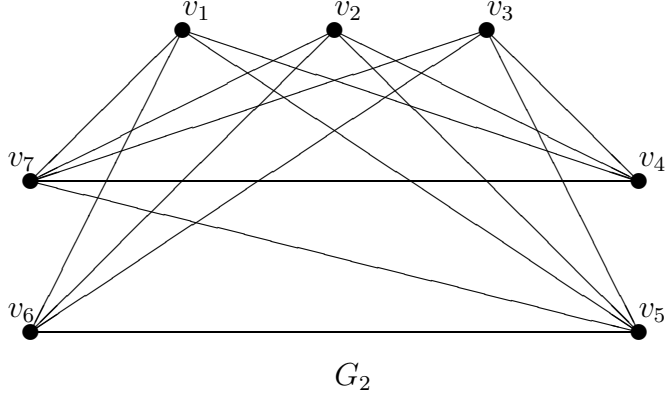
\begin{figure}[ht]
\setlength{\unitlength}{1mm}
\begin{picture}(120, 60)
\centering \put(60,50){\circle*{2}}
   \put(40,50){\circle*{2}}
\put(80,50){\circle*{2}}

\put(20,30){\circle*{2}}\put(20,10){\circle*{2}}

\put(100,30){\circle*{2}}\put(100,10){\circle*{2}}

\put(40,50){\line(-1,-1){20}} \put(40,50){\line(-1,-2){20}}
\put(40,50){\line(3,-1){60}} \put(40,50){\line(3,-2){60}}

 \put(60,50){\line(-2,-1){40}}
\put(60,50){\line(-1,-1){40}}

 \put(60,50){\line(2,-1){40}}
\put(60,50){\line(1,-1){40}}

\put(80,50){\line(1,-1){20}} \put(80,50){\line(1,-2){20}}
\put(80,50){\line(-3,-1){60}} \put(80,50){\line(-3,-2){60}}

\put(20,30){\line(1,0){80}} \put(20,30){\line(4,-1){80}}
 \put(20,10){\line(1,0){80}}

\put(40,52){$v_{1}$} \put(60,52){$v_{2}$} \put(80,52){$v_{3}$}

\put(100,32){$v_{4}$} \put(100,12){$v_{5}$}
 \put(17,12){$v_{6}$}
\put(17,32){$v_{7}$}  \put(60,3){$G_2$}

\end{picture}
\end{figure}
Then we have following results \vskip 0.5cm

\begin{tabular}{|c|c|c|c|c|c|c|}
\hline  & $q_1(G)$ &  (\ref{thm1-1}) & (\ref{Oli158}) &
(\ref{Oli158-2}) &(\ref{Liucmj})& (\ref{shu434})\\
 \hline
   $T_{10,3}$    & 13.2915 &13.2915 & 13.7082 & 13.6119 &13.6667  & 13.5826\\
\hline $G_2$ &8.7417 & 9.2749 &9.5826 &9.4462 &9.6667  &8.8284\\
\hline
\end{tabular}
\end{example}

\vskip2cm
\begin{center} \vskip 0.3cm
 {\bf Acknowledgements}
\end{center}

The authors wish to thank the referees for their valuable comments
and suggestions.

 The  third author  would like to express his deepest
thanks to his advisor Professor Abraham Berman for introducing him
the beauty  he found in Linear Algebra and Combinatorics (Lady Davis
Postdoctoral Fellowship from Oct. 1998  to Aut. 2000 in Technion).
He is grateful to Professor Moshe Goldberg, and Raphael Loewy  for
their support and encouragement.

Added in proof. The authors are grateful to Professor V.~Nikiforov
for pointing out that  a similar result was recently  and
independently obtained by N.~M.~M. de Abreu and him  using different
methods.

\end{document}